\documentclass{amsart} 

\usepackage{amssymb}
\usepackage{comment}
\usepackage{amsmath}
\usepackage{latexsym}
\usepackage{amsbsy}
\usepackage{amsfonts}
\usepackage{hyperref}
\usepackage{graphicx}
\usepackage{enumerate}
\usepackage{color}

\usepackage{tikz}

\def\marginpar#1{\ignorespaces}

\newtheorem{theorem}{Theorem}[section]

\newtheorem{corollary}[theorem]{Corollary}

\numberwithin{equation}{section}
\makeatother
\newcommand{\TITLE}[1]{\renewcommand{\@TITLE}{#1}}
\newcommand{\SHORTTITLE}[1]{\renewcommand{\@SHORTTITLE}{#1}}
\newcommand{\DEDICATORY}[1]{\gdef\@DEDICATORY{#1}}

\newcommand{\KEYWORDS}[1]{\renewcommand{\@KEYWORDS}{#1}}
\newcommand{\AMSSUBJ}[1]{\renewcommand{\@AMSSUBJ}{#1}}
\newcommand{\AMSSUBJSECONDARY}[1]{\gdef\@AMSSUBJSECONDARY{#1}}
\newcommand{\ABSTRACT}[1]{\renewcommand{\@ABSTRACT}{#1}}
\newcommand{\VOLUME}[1]{\renewcommand{\@VOLUME}{#1}}
\newcommand{\PAPERNUM}[1]{\renewcommand{\@PAPERNUM}{#1}}
\newcommand{\YEAR}[1]{\renewcommand{\@YEAR}{#1}}
\newcommand{\PAGESTART}[1]{\renewcommand{\@PAGESTART}{#1}} 
\newcommand{\PAGEEND}[1]{\renewcommand{\@PAGEEND}{#1}} 
\newcommand{\SUBMITTED}[1]{\renewcommand{\@SUBMITTED}{#1}}
\newcommand{\ACCEPTED}[1]{\renewcommand{\@ACCEPTED}{#1}}
\newcommand{\DOI}[1]{\gdef\@DOI{10.1214/\@JOURNAL.#1}}
\newcommand{\ARXIVID}[1]{\gdef\@ARXIVID{#1}}
\newcommand{\HALID}[1]{\gdef\@HALID{#1}}

\def\EMAIL#1{E-mail:~\texttt{\href{mailto:#1}{\nolinkurl{#1}}}}

\renewcommand{\P}{\mathbb{P}}

\newcommand{\E}{\mathbb{E}}

\newcommand{\convd}{\overset{(d)}{{}\rightarrow{}}}

\newcommand{\giv}{\,|\,}  

\begin{document}
\title[Feller coupling of cycles of permutations]{Feller coupling of cycles of permutations and Poisson spacings in inhomogeneous Bernoulli trials}


\author[Joseph Najnudel]{{Joseph} Najnudel}
\address{School of Mathematics, University of Bristol, UK.}
\email{joseph.najnudel@bristol.ac.uk}
\author[Jim Pitman]{{Jim} Pitman}
\address{Statistics department, University of California, Berkeley, USA.} \email{pitman@stat.berkeley.edu}
\begin{abstract}
{Feller (1945) provided a coupling between the counts of cycles of various sizes in a uniform random
permutation of $[n]$ and the spacings between successes in a sequence of $n$ independent Bernoulli trials with success probability $1/n$ at the $n$th trial. 
Arratia, Barbour and Tavar\'e (1992) extended Feller's coupling, to associate cycles of random permutations governed by the Ewens $(\theta)$ distribution with 
spacings derived from independent Bernoulli trials with success probability $\theta/(n-1+\theta)$ at the $n$th trial, and to conclude that in an infinite sequence of such trials, the numbers of spacings of length $\ell$ are independent Poisson variables with 
means $\theta/\ell$.  
Ignatov (1978) first discovered this remarkable result in the uniform case $\theta = 1$, by constructing Bernoulli $(1/n)$ trials as the indicators of record  values in a sequence of i.i.d. uniform $[0,1]$ variables.
In the present article, the Poisson property of inhomogeneous Bernoulli spacings is explained by a variation of Ignatov's approach for a general $\theta >0$.
Moreover, our approach naturally provides random permutations of infinite sets whose cycle counts are exactly given by independent Poisson random variables.}
\end{abstract}
\maketitle

\section{Introduction}

In \cite{feller1945}, Feller introduces a coupling between the cycle structure of a uniformly distributed random permutation of order $n$ and the spacings between 
successes in a sequence of $n$ independent Bernoulli variables of parameters $1/n, 1/(n-1), \dots, 1/2, 1$. 
This coupling has been generalized to Ewens distributions for any parameter $\theta$: a recent discussion on this topic, with references to further work is provided by
Arratia, Barbour,  and Tavar\'{e}
\cite{MR3458588}, largely following their earlier work 
\cite{MR1177897}. Their coupling, for a general positive integer $n$ and $\theta >0$, may be constructed as follows. Consider a sequence $(B_i(\theta))_{1 \leq i \leq n}$ of independent Bernoulli variables, 
$B_i(\theta)$ with parameters $\theta/(\theta+ i -1)$. 
Conditionally on $(B_i(\theta))_{1 \leq i \leq n}$, construct the random permutation $\sigma$
of the set
 $[n]:= \{1,2, \dots, n\}$, as follows. First, define $X_1 := 1$, and then, recursively for $2 \leq i \leq n$: 
 \begin{itemize}
 \item If $B_{n+2-i}(\theta) = 1$, $X_i$ is the smallest element of $[n]$, different from $X_1, \dots, X_{i-1}$. 
 \item Conditionally on the fact that $B_{n+2-i}(\theta)= 0$, and on the values of $X_1, \dots, X_{i-1}$, the element $X_i$ is uniformly distributed on $[n] \backslash \{X_1, \dots, X_{i-1} \}$. 
 \end{itemize}
Then, the cycle structure of the permutation $\sigma$ is obtained by taking the subsequences of $(X_1, X_2, \dots, X_n)$, in such a way that 
the value $X_i$ is the start of a cycle if and only if $i = 1$ or $B_{n+2-i}(\theta) = 1$.

For example, suppose $n=9$, 
$$(B_1(\theta), B_2 (\theta), \dots, B_9(\theta)) = (1, 0, 1, 0, 0, 1, 1, 0, 0).$$
  A possible realization of the $X_i$'s is 
  $$(X_1, X_2, \dots, X_9) = (1,7,3,2,4,9,5,6,8).$$
  Since $B_3(\theta)$, $B_6(\theta)$ and $B_7(\theta)$ are equal to $1$, we have cycles starting at $X_8, X_5$ and $X_4$, and then 
$$\sigma = (173)(2)(495)(68).$$
Note that written in this fashion, each cycle starts with its minimal element, and the cycles are
written in increasing order of their minimal elements.

To indicate the parameters $n$ and $\theta$ used in this construction, let $\pi_{n, \theta}$ denote the random permutation $\sigma$ of $[n]$
so constructed.
Then $\pi_{n, \theta}$ 
follows the Ewens distribution
$$
\P( \pi_{n,\theta} = \pi ) = \frac{ \theta^{K(\pi)}}{ (\theta)_{n} } \mbox{ where } (\theta)_n:= \theta (\theta + 1) \cdots (\theta + n - 1)
$$
and $K(\pi)$ is the number of cycles of a permutation $\pi$ of $[n]$.
The proof of this fact is indicated in \cite{MR3458588} and appeals to Feller's original coupling
 of $B_1(1), \ldots, B_n(1)$ to a uniform random permutation $\pi_{n,1}$ for $\theta = 1$, and a simple change of measure argument for $\theta \ne 1$. 

The cycle structure of $\pi_{n, \theta}$ can be deduced from the spacings between the Bernoulli variables $B_i(\theta)$ which are equal to $1$. 
More precisely, for $\ell \ge 1$, let us say that an $\ell$-spacing occurs in a sequence $a_1, a_2, \ldots$ of $0$s and $1$s, starting at position $i-\ell$ and ending at
position $i$, if
$$
a_{i-\ell} \cdots a_i = 1\, 0^{\ell-1} \, 1
$$
meaning that the string of length $\ell + 1$ is a $1$ followed by $\ell - 1$ zeros followed by $1$.
If $C_{n,\ell}(\theta)$ is the number of $\ell$-spacings in
$$
B_1(\theta), \ldots, B_n(\theta),1, 0,0,0, \ldots
$$
then there is the equality
\begin{equation}
(C_{n,\ell}(\theta), 1 \le \ell \le n) = (K_{\ell}(\pi_{n,\theta}), 1 \le \ell \le n) 
\label{ckidx}
\end{equation}
where $K_{\ell}(\pi_{n,\theta})$ is 
the number of cycles of length $\ell$ in the permutation $\pi_{n,\theta}$. 

By regarding the sequence $(B_i(\theta))_{1 \leq i \leq n}$ as the first $n$ terms of an infinite sequence $(B_i(\theta))_{i \geq 1}$ of independent Bernoulli variables, 
 we get a coupling, on a single probability space, of the families of cycle lengths $(K_{\ell}(\pi_{n,\theta}))_{1 \le \ell \le n}$ for all values of $n$.  
We quickly deduce the following result by Arratia, Barbour and Tavar\'e, for which we provide a sketch of proof here for the reader's convenience in comparing with later arguments:
\begin{theorem} \label{11}
{\em Arratia, Barbour and Tavar\'e (\cite{MR1177897},
\cite{MR3458588}) }
 If $C_{\infty,\ell}(\theta)$ is the number of $\ell$-spacings in the infinite sequence $(B_i(\theta))_{i \geq 1}$ of Bernoulli variables, $B_i(\theta)$ having parameter $\theta/(\theta+ i -1)$,  then $(C_{\infty,\ell}(\theta))_{\ell \geq 1}$ is a sequence of independent Poisson$(\theta/\ell)$ variables.
\end{theorem}
\begin{proof}
If $L_n(\theta)$ is the position of the last $1$ in $(B_i(\theta))_{ 1 \le i \le n }$ and $J_n(\theta):= n+1 - L_n(\theta)$ is the last spacing in the finite $n$ scheme, then
\begin{equation}
\label{coupling}
C_{n,\ell}(\theta) \le C_{\infty,\ell}(\theta) + 1 ( J_n(\theta) = \ell ) \qquad (1 \le \ell \le n)
\end{equation}
with strict inequality iff there is an $\ell$-spacing in the infinite sequence $(B_j(\theta))_{j \geq 1}$ starting at $j = i -\ell$ and ending at $j = i > n$.
Now, this event and the event $\{J_n(\theta) = \ell \}$ have probability tending to zero when $n \rightarrow \infty$, so for fixed $\ell$,  $C_{n,\ell}(\theta) = C_{\infty,\ell}(\theta) $ with probability tending to one. On the other hand, by \eqref{ckidx}, for any fixed $K \geq 1$,  $(C_{n,\ell}(\theta), 1 \le \ell \le K)$ tends in law to independent  Poisson$(\theta/\ell)$ variables. The two last facts together imply the theorem. 
\end{proof}
Combining  \eqref{ckidx} and \eqref{coupling}, we get a coupling  of counts of small cycles 
$(K_\ell(\pi_{n,\theta}))_{1 \le \ell \le k}$ of a Ewens$(\theta)$ permutation to independent Poisson $(\theta/\ell)$ counts $(C_{\infty,\ell}(\theta))_{1 \le \ell \le k}$,
with a total variation error depending on $k$ and $\theta$ which is easily bounded explicitly. This implies in particular that 
\begin{equation}
\label{convd}
( K_\ell(\pi_{n,\theta}), 1 \le \ell \le k) \convd (C_{\infty,\ell}(\theta), 1 \le \ell \le k) \mbox{ as } n \to \infty
\end{equation}
for every fixed $k$, as well as estimates of total variation error in this approximation which are useful for $k = o(n)$:
see \cite[Theorems 1 and 3]{MR1177897}.
See also Sethuraman and Sethuraman \cite{MR2744276} 
for a review of studies of the distribution of the numbers of $\ell$-spacings in infinite sequences of independent Bernoulli trials with sequences of probabilities $p_i$ other than the sequence $p_i = \theta/(\theta + i -1)$
involved in this coupling with a sequence of Ewens$(\theta)$ permutations.

The coupling described above provides a way to define a sequence of Ewens$(\theta)$ random permutations $(\pi_{n, \theta})_{n \geq 1}$ whose cycle structures for different values of $n$ are strongly related: 
from $\pi_{n,\theta}$ to $\pi_{n+1,\theta}$, either a single fixed point is added,
or a single cycle of $\pi_{n,\theta}$ has its length increased by one.
However, the coupling above does not uniquely define a joint distribution for $\pi_{n, \theta}$ and $\pi_{n+1, \theta}$, because it does not say how
the content of the cycles of $\pi_{n+1, \theta}$ and $\pi_{n, \theta}$ are related.

In the particular case $\theta = 1$, when each $\pi_{n,1}$ is a uniform random element of the set
$\mathfrak{S}_n$ of permutations of $[n]$, Ignatov \cite{MR617405} provides a nice construction which defines the joint law of $(\pi_{n, 1})_{n \geq 1}$ in a unique way. 
Let $(U_i)_{i \geq 1}$ 
be
a sequence of pairwise distinct elements of $[0,1]$, with no smallest element. 
From this sequence, define the {\it lower record indices} $I_1 < I_2 < I_3 < \dots$, 
 as the set of indices $I$  such that $U_I$ is smaller than $U_i$ for all $i < I$, the {\it lower indicators} $(B_i)_{i \geq 1}$, given by $B_i = 1$ if $i$ is a lower record index 
 and by $B_i = 0$ otherwise, and the 
{\em inter-record stretches} $(V_k)_{k \geq 1}$ given by: 
\begin{equation}
V_k:= (U_{I_k}, U_{I_k+1}, \dots, U_{I_{k+1}-1}).
\end{equation}
We notice the following facts:
\begin{itemize}
\item the inter-record stretches are elements of the space $\cup_{\ell = 1}^\infty [0,1]^\ell$ of finite sequences in $[0,1]$ with undetermined length; 
\item the first term of the stretch $V_k$ is the $k$-th lower record value $R_k := U_{I_k}$;
\item this first term $R_k$ of $V_k$ is  the minimal term of the stretch $V_k$;
\item the length of the stretch $V_k$ is $I_{k+1} - I_k$, the $k$-th {\em inter-record spacing}.
\end{itemize}

We can then define, for all $n \geq 1$, a permutation $\pi^{U}_n$ of $\{U_1, \dots, U_n\}$ whose cycle structure is given by the 
inter-record stretches:  more precisely, $\pi^U_n (U_{i-1}) = U_{i}$ for all $i \in \{2, \dots, n\}$ which are not lower record indices, $\pi^U_n (U_{I_{k+1} - 1}) = U_{I_k}$ if 
$k \geq 1$ is such that $I_{k+1}  -1 \leq n$, and $\pi^U_n (U_n) = U_{I_j}$ where $I_j$ is the last lower record index such that $I_j \leq n$. 
The permutation $\pi^U_n$ acts on the set $\{U_1, \dots, U_n\}$: it induces a permutation $\pi_n$ of $[n]$ if we rename the $m$-th smallest element 
of this set by $m$ (for example, the permutation $0.2 \mapsto 0.9$, $0.4 \mapsto 0.5$, $0.5 \mapsto 0.2$, $0.9 \mapsto 0.4$ induces the permutation 
$1 \mapsto 4$, $2 \mapsto 3$, $3 \mapsto 1$, $4 \mapsto 2$). 
It is not difficult to check that the permutation $\pi_n$ depends only on the relative order of $U_1, \dots, U_n$, in a way which induces a bijective map
from $\mathfrak{S}_n$ to itself. This bijection was proposed by R\'enyi 
\cite{MR0286162} in the early 60s, and called the ``transformation fondamentale'' in a paper by Foata and Sch\"utzenberger  \cite{MR0272642}, in a more general setting of combinatorics on words.
Diaconis and Pitman \cite{diac-pit86} exploited this bijection to obtain the convergence in distribution \eqref{convd} in the case $\theta = 1$,
with a total variation bound. 
This bound was sharpened and extended to the case of a general parameter $\theta >0$ in \cite{MR1177897}, as indicated above.
But this argument for general $\theta$ loses track of the full Poisson structure of the record process for $\theta = 1$.

Let us recall how this Poisson structure for $\theta = 1$ was first exposed by Ignatov \cite{MR617405}.
If $(U_i)_{i \geq 1}$ is a sequence of i.i.d., uniform variables on $[0,1]$, then for all $n \geq 1$, all possible orders of $U_1, \dots, U_n$ occur with the same probability and 
then $\pi_n$ is uniformly distributed on $\mathfrak{S}_n$. On the other hand, the lower record indicators $(B_i)_{i \geq 1}$ are independent, Bernoulli variables, $B_i$ having parameter $1/i$. 
The link between the construction of $\pi_n$ and the Feller coupling is the following: conditionally on $(B_i)_{i \geq 1}$, 
the distribution of $\pi_n$ is uniform on the set of permutations whose lengths of the cycles, ordered by increasing lowest element, are equal  to the successive spacings between 
the $1$'s in the sequence $(1,B_n, B_{n-1}, \dots, B_1)$. One easily deduces the following result: given $(B_i)_{1 
\leq i \leq n}$, the
conditional distribution of $\pi_n$  is the same as that given by the Feller coupling procedure
using $B_i(1) = B_i$ for all $i$.
The Poisson structure obtained by taking all the inter-record stretches together is quite remarkable:

\begin{theorem} 
\label{thm:ignatov}
If the variables $(U_i)_{i \geq 1}$ are i.i.d., uniform in $[0,1]$, then the inter-record stretches $(V_k)_{k \geq 1}$ form a Poisson point process on $\cup_{\ell = 1}^\infty [0,1]^\ell$ with
mean measure
\begin{equation}
\label{mudef}
\mu (\bullet) = \sum_{\ell = 1}^\infty \frac{ P_\ell (\bullet) } { \ell }
\end{equation}
where $P_\ell(\bullet)$ is the conditional
distribution of $(U_1, \ldots, U_\ell)$ given that $U_1 < U_i$ for every $1 < i \le \ell$.
\end{theorem}
To illustrate the notation:
\begin{itemize} 
\item $P_1(\bullet)$ is the uniform distribution of $U_1$ on $[0,1]$,
\item $P_2(\bullet)$ is uniform on $ \{ (u_1,u_2): u_1 < u_2 \} \subseteq [0,1]^2$, that is
the conditional distribution of $(U_1,U_2)$ given the event $(U_1 < U_2)$ of probability $1/2$.
\item $P_3(\bullet)$ is uniform on $ \{ (u_1,u_2,u_3): u_1 < \min( u_2, u_3 ) \} \subseteq [0,1]^3$, that is
the conditional distribution of $(U_1,U_2,U_3)$ given the event $(U_1 < \min ( U_2, U_3))$ of probability $1/3$,
\end{itemize} 
and so on.

Theorem \ref{thm:ignatov} is a straightforward extension  of the result of 
Ignatov \cite{MR617405}  that $\{ (R_k, I_{k+1} - I_k ), k \ge 1 \}$ is the collection of points of a Poisson point process on $(0,1) \times \{1,2, \ldots \}$
with mean number of points in $(s,1] \times \{ \ell \}$  equal to  $(1-s)^\ell /\ell$.    See Resnick \cite[Proposition 4.1 (iv)] {MR900810} for a
proof of this result using the basic spraying property of Poisson processes \cite[Proposition 3.8] {MR900810}.  The same spraying argument
gives the stronger assertion of Theorem \ref{thm:ignatov}. For it is easily seen that given all the  points $\{ (R_k ,I_{k+1} - I_k ), k \ge 1 \}$,  
for each particular $k$, the conditional distribution of the stretch $V_k$ depends only on $R_k$ and $I_{k+1} - I_k$, and given $R_k = r$ and $I_{k+1} - I_k  = \ell$, the stretch $V_k$ with initial term $r$ and length $\ell$ has the distribution of $(U_1, \ldots, U_\ell)$ given
$U_1 = r$ and  $r < U_i $ for all $1 < i \le \ell$. 

If we only consider the length of the inter-record stretches, we immediately deduce from Theorem \ref{thm:ignatov} that 
the inter-record spacings $(I_{k+1} - I_k)_{k \geq 1}$, form a Poisson point process on the positive integers, with intensity $1/\ell$ at $\ell$, meaning that the random variables
\begin{equation}
\label{kinfl}
K_{\infty,\ell}:= \sum_{k=1}^\infty 1 ( I_{k+1} - I_k = \ell)
\end{equation}
are independent Poisson variables with means $1/\ell$. This results corresponds to the case $\theta =1 $ of Theorem  \ref{11}, since the lower record indicators $(B_i)_{i \geq 1}$ are independent, Bernoulli variables, $B_i$ having parameter $1/i$.

Regarded as a fact about inhomogeneous Bernoulli trials, this result
is not at all obvious without a broader context involving additional randomization, such as Ignatov's context of record sequences, or
the context of the Feller coupling for random permutations.

The link between the Feller coupling and the lower records of a sequence of random variables can be extended to the setting of Ewens distributed permutations with general parameter $\theta > 0$, 
by changing the distribution of the sequence $(U_i)_{i \geq 1}$.
We will prove the following result: 
\begin{theorem} 
\label{thm:np}
For  $\theta >0$, let  $\P_\theta$ be the probability measure on the set of infinite sequences  in $[0,1]$, endowed with its Borel $\sigma$-algebra, such that 
for  $(U_i)_{i \geq 1}$ following the law $\P_{\theta}$: 
\begin{itemize}
\item The first term $U_1$ is Beta distributed with parameters $\theta$ and $1$. 
\item Conditionally on $(U_1, \dots, U_n)$, for $\min(U_1, \dots, U_n) = r$, the distribution of $U_{n+1}$ is the mixture with weights
$r$ and $1-r$ of the distribution of $r$ times a Beta variable with parameters $\theta$ and $1$, and the uniform distribution on $[r,1]$, i.e.
$$ \frac{\P_\theta( U_{n+1} \in du  \giv \min(U_1, \dots, U_n)   = r )}{du} = \theta \, \left( \frac{u}{r} \right)^{\theta - 1} 1(u < r) + 1 (u \ge r) \qquad ( 0 < u < 1 ).$$
\end{itemize}
In particular, no matter what $\theta >0$, the conditional probability of a new lower record at time $n+1$, given $(U_1, \dots, U_n)$, is always 
$\min(U_1, \dots, U_n)$. 

Then, the following statements hold:
\begin{itemize}
\item The finite dimensional distributions of $(U_1, \ldots, U_n)$ under $\P_\theta$ are absolutely continuous with respect to the
Lebesgue measure on $[0,1]^n$, 
with density
\begin{equation}
\label{rnd1}
\frac{d \P_\theta}{ d \P_1} (u_1, \ldots, u_n) =  \theta^{K_n} \min(u_{1}, \dots, u_n)^{\theta -1} 
\end{equation}
where $K_n$ is the number of lower records in the sequence $(u_1, \ldots, u_n)$. In particular, under $\P_1$, the variables $(U_i)_{i \geq 1}$ are i.i.d., uniform on $[0,1]$.
\item If   $(U_i)_{i \geq 1}$ follows the law $\P_{\theta}$, then  this sequence has a.s. no smallest element, the $U_i$'s are pairwise distinct, and 
the inter-record stretches $(V_k)_{k \geq 1}$ form a Poisson point process on $\cup_{\ell = 1}^\infty [0,1]^\ell$ with
mean measure $\theta \mu(\bullet)$ for $\mu(\bullet)$ as in \eqref{mudef}.
\end{itemize}

\end{theorem}
The fact that $(U_i)_{i \geq 1}$ are i.i.d., uniform   under $\P_1$ is a restatement of Ignatov's Theorem \ref{thm:ignatov}.
 The description of the law of 
   $(U_1, \ldots, U_n)$ under $\P_\theta$ for general $\theta$  has already been indicated by Kerov and Tsilevich 
\cite[Lemma 2]{ktstick}, with upper rather than lower records, which exchanges $U_i$ with $1-U_i$ in the formulas. 
Kerov and Tsilevich have also associated random permutations to sequences following the distribution $\P_{\theta}$, and these permutations are distributed with respect to the Ewens measure of parameter $\theta$. However, the construction of \cite{ktstick} does not coincide with the construction given in the present paper 

The fact that $\P_\theta$ may also be described as in Theorem \ref{thm:np}, by simply changing the mean intensity measure 
of the Poisson point process of inter-record stretches on $\cup_{\ell = 1}^\infty[0,1]^\ell$ by a scalar factor of $\theta$,  from $\mu(\bullet)$ under $\P_1$ to $\theta \mu(\bullet)$ under $\P_\theta$, does not seem to have been observed before. 

The push forward of this result, from the Poisson point process of inter-record stretches 
to the Poisson point process of their lengths, gives the fact that the counting of the 
 inter-record spacings $(I_{k+1} - I_k)_{k \geq 1}$ forms a Poisson point process on the positive integers, with intensity $\theta/\ell$ at $\ell$. 
 The following corollary links  Theorem \ref{thm:np} to the Feller coupling and explains how Theorem \ref{thm:np} implies Theorem \ref{11}.
 \begin{corollary} \label{crl:feller}
For a sequence $(U_i)_{i \ge 1}$ of pairwise distinct elements of $[0,1]$, with no smallest element,
let $(B_i)_{i \geq 1}$ be the corresponding lower record indicators, let
$\pi_n^U$ be the permutation of  $\{U_1, \ldots, U_n \}$ whose cycle structure is given by the inter-record stretches, and let $\pi_n$ be the
corresponding permutation of $[n]$, with the same cycle structure.
Then,  for $(U_i)_{ i \ge 1}$ governed by the law $\P_\theta$,
\begin{itemize}
\item The $B_i$ are independent Bernoulli $(\theta/(i-1+ \theta))$;
\item Conditionally on all the $B_i$, the permutation $\pi_n$ is uniformly distributed 
among the permutations whose cycle lengths, ordered by increasing lowest elements, are given by the successive spacings between $1$'s in the sequence $(1, B_n, B_{n-1}, \dots, B_1)$;
\item Given $(B_i)_{1 
\leq i \leq n}$, the
conditional distribution of $\pi_n$  is the same as that given by the Feller coupling procedure
using $B_i(\theta) = B_i$ for all $i$;
\item The unconditional distribution of $\pi_n$ is Ewens with parameter $\theta$.
\end{itemize}
\end{corollary}

Another thing we can remark is that all the values of $\theta > 0$ can be coupled on a single probability space. Indeed, under $\mathbb{P}_{\theta}$, the family of inter-record stretches forms a Poisson point
process of intensity $\theta \mu(\bullet)$, so can be constructed simultaneously for all $\theta$ by taking the points of a Poisson process of intensity equal to the product of Lebesgue measure on $\mathbb{R}_+$
by the measure $\mu(\bullet)$, and extracting the points for which the  $\mathbb{R}_+$  coordinate is smaller than $\theta$. Such a coupling provides a dynamic version of the Feller coupling, with the parameter 
$\theta$ of the Ewens measure as its time parameter. The path structure of this $\mathfrak{S}_n$-valued process $(\pi_{n,\theta}, \theta \ge 0)$ can be understood as follows. It may be constructed with right-continuous
step function paths, in which each jump involves insertion of a new cycle of some length $\ell$ from $1$ to $n$, corresponding to a Poisson point which is a sequence in 
some component $[0,1]^k$ of the sequence space with $k \ge \ell$, whose initial term is greater than the initial term of at least one sequence contributing to the current permutation of $[n]$.
This insertion may delete some cycles, and/or shorten  the final cycle, depending on the rank of the initial term of the new sequence relative to the initial terms associated with existing cycles.
It does not seem easy to give a full probabilistic description of the dynamics of this $\mathfrak{S}_n$-valued process.
In particular, it may not be Markovian, due to the latent initial terms of the sequential fragments which determine the order of the cycles.  
As the partition of $n$ induced by $\pi_{n,\theta}$ is not necessarily refining as $\theta$ increases, this process is not the same as the 
evolution described by Gnedin and Pitman \cite{MR2351686}, in which partitions following the Ewens $(\theta)$ distribution are constructed for all values of $\theta > 0$ to be refining as $\theta$ increases.
 
Theorem \ref{thm:np} and  Corollary \ref{crl:feller} are proven in Section \ref{thmnp} of the present article. 
In Section \ref{infinite}, we use the measure $\mathbb{P}_{\theta}$ in order to construct some infinite random permutations, in  a way which generalizes the Feller coupling.
In Section \ref{lloyd}, we provide a link between our construction and a result by Shepp and Lloyd on the cycle counts of permutations of random order.

\section{Proof of Theorem \ref{thm:np} and  Corollary \ref{crl:feller} } \label{thmnp}
By induction on $n$, using the definition of  $\P_{\theta}$, we see that the 
density at $(u_1, \dots, u_n)$ of the law of $(U_1, \dots, U_n)$ under $\P_{\theta}$ is given by 
$$\theta  u_1^{\theta-1} \prod_{2 \leq j \leq n \,: \, u_j = \min (u_1, \dots, u_j)}  \left[ \theta \left( \frac{ \min(u_1, \dots, u_j)}{\min(u_1, \dots, u_{j-1})} \right)^{\theta-1} \right]
= \theta^{K_n} (\min(u_1, \dots, u_n))^{\theta-1},$$
which proves the first statement of the theorem. 
It is also clear from the definition that  $(U_i)_{i \geq 1}$ has a.s. no smallest element and all elements pairwise distinct. 

Now, for a given $s \in (0,1)$, $n \geq 0$ and $u_1, \dots, u_n \in (s,1)$, we easily check that 
\begin{align*}
& \frac{\P_{\theta} (U_1 \in du_1, \dots, U_n \in du_n, U_{n+1} \leq s)}{du_1 \dots du_n} \\
& = \theta^{K_n}  (\min(u_1, \dots, u_n))^{\theta-1}  \int_0^s  \left[ \theta \left( \frac{u}{\min(u_1, \dots, u_{n})} \right)^{\theta-1} \right] du
\\ & = \theta^{K_n} \int_0^s \theta u^{\theta -1} du = \theta^{K_n} s^{\theta}.
\end{align*}
Let $(\mathcal{F}_i)_{i \geq 0}$ be the filtration generated by the variables $(U_i)_{i \geq 1}$, and let $T_s$ be the first index $i$ such that $U_i \leq s$: it is 
clear that $T_s$ is a stopping time with respect to $(\mathcal{F}_i)_{i \geq 0}$. 
The equality above shows that for any event $A_n$ which is $\mathcal{F}_n$-measurable, 
$$\P_{\theta} (A_n, T_s = n+1) = s^{\theta-1} \mathbb{E}_{\mathbb{P}_1} [ 1(A_n, T_s = n+1) \, \theta^{K_n} ]$$
where $K_n$ is the number of lower records in the sequence $(U_1, \dots, U_n)$. 
Now, let $E$ be an event which is measurable with respect to the family of all inter-record stretches starting above the level $s$, and let $L$ be the 
total length of these stretches. We can check that for $n \geq 0$, the intersection of $E$ and the event $\{L = n\}$ can be written as the intersection of 
$A_n$ and $T_s = n+1$ for some $\mathcal{F}_n$-measurable event $A_n$, which gives 
$$\P_{\theta} (E, L=n) = s^{\theta-1} \mathbb{E}_{\mathbb{P}_1} [ 1(E, L=n) \, \theta^{N_s} ]$$
where $N_s$ is the number of inter-record stretches starting above the level $s$. 
Hence
$$\P_{\theta} (E) = s^{\theta-1} \mathbb{E}_{\mathbb{P}_1} [ 1(E) \, \theta^{N_s} ],$$
which implies the following: 
\begin{itemize}
\item The law of $N_s$ under $\mathbb{P}_{\theta}$ has density $\theta^{N_s} s^{\theta-1}$ with respect to the law of $N_s$ under $\mathbb{P}_1$. 
\item Conditionally on $N_s$, the set of inter-record stretches starting above the level $s$ has the same law under $\mathbb{P}_{\theta}$ and under $\mathbb{P}_1$. 
\end{itemize}
By Theorem \ref{thm:ignatov}, the law of $N_s$ under $\mathbb{P}_1$ is the  Poisson distribution with parameter 
$$\sum_{\ell = 1}^{\infty} \frac{P_{\ell} ([s, 1]^{\ell})}{\ell} = \sum_{\ell=1}^{\infty} \frac{(1-s)^{\ell}}{\ell} = - \log s$$
and we deduce that the law of $N_s$ under $\mathbb{P}_{\theta}$ is the Poisson distribution of parameter $-\theta \log s$. 
Moreover, conditionally on $N_s$, under $\mathbb{P}_1$, and then also under $\mathbb{P}_{\theta}$, the family of inter-record stretches starting above $s$ has the same law as 
the family of elements of an  i.i.d. sequence of variables which are distributed according to the 
probability measure: 
$$-\frac{1}{\log s} \sum_{\ell} \frac{(P_{\ell} )_{| [s,1]^{\ell}} } {\ell}.$$
Hence, under $\mathbb{P}_{\theta}$, the inter-record stretches starting above $s$  form a Poisson point process 
with intensity 
$$\theta \sum_{\ell} \frac{(P_{\ell} )_{| [s,1]^{\ell}} } {\ell}.$$
Since $s \in (0,1)$ can be arbitrarily chosen, we get the second statement of Theorem \ref{thm:np}.

 For the corollaries, we use the following key property: the density on $\mathcal{F}_n$ of $\mathbb{P}_{\theta}$ with respect to $\mathbb{P}_1$ can 
 be written as the product of a function of the relative order of $(U_1, \dots, U_n)$, i.e.  $\theta^{K_n}$, and a function of the order statistics 
 of $(U_1, \dots, U_n)$, i.e. $(\min(U_1, \dots, U_n))^{\theta-1}$.  
 Since the relative order and the order statistics of $(U_1, \dots, U_n)$ are independent under $\mathbb{P}_1$, they remain independent under $\mathbb{P}_{\theta}$ for all $\theta > 0$. 
Moreover, under $\P_1$, the record indicators $B_i$ are independent Bernoulli$(1/i)$ variables, and the change of measure on these variables when we go from  $\mathbb{P}_{1}$
to $\mathbb{P}_{\theta}$  corresponds to a density factor proportional to
$$
\theta^{K_n} = \prod_{i=1}^n \theta^{B_i}.
$$
This easily implies that under $\P_\theta$, the record indicators are independent Bernoulli$(p_i(\theta))$, where
$$
p_i (\theta) = \frac{ \theta/i} { 1 - 1/i + \theta/i } = \frac{ \theta}{ i - 1 + \theta} .
$$
This gives the first item of Corollary \ref{crl:feller}. 
Moreover, conditionally on the  lower record indicators and the order statistics of $(U_1, \dots, U_n)$, all the possible 
relative orders of $(U_1, \dots, U_n)$ have the same probability, because of the form of the density of $\mathbb{P}_{\theta}$ with respect to $\mathbb{P}_1$. 
This  implies the second item of Corollary \ref{crl:feller}, since
the permutation $\pi_n$ is uniquely determined by the relative order of $(U_1, \dots, U_n)$. 
The third item of Corollary \ref{crl:feller} is a direct consequence of the two first items, and the last item is due to the classical properties of the Feller coupling. 

\section{Infinite permutations} \label{infinite}
From any sequence $(U_i)_{i \geq 1}$ of elements in $[0,1]$, with distinct values and  no smallest element, we have seen how to construct 
a permutation $\pi^U_n$ of $\{U_1, \dots, U_n\}$ and a permutation $\pi_n$ of $[n]$ from the inter-record stretches. 
It is also possible to define a permutation $\pi^U_{\infty}$ of the infinite set $\{U_i, i \geq 1\}$, in such a way that the cycles are given by the set of all inter-record stretches, 
i.e. $\pi^U_{\infty} (U_{i-1}) = U_i$ for all $i \geq 2$  which are not lower record indices, and $\pi^U_{\infty} (U_{I_{k+1} - 1}) = U_{I_k}$ for all $k \geq 1$. 
One easily checks that $\pi^U_{\infty}$ coincides with $\pi^U_n$ on the set $\{U_1, \dots, U_{n-1}\}$ for all $n \geq 1$. 

The construction of $\pi^U_{\infty}$ can be seen as some kind of Feller coupling of infinite order, since the construction of $\pi^U_n$ and  $\pi_n$ can be related with the Feller coupling 
of order $n$, as we have seen previously. 
However, we observe that contrary to the case of the permutation $\pi_n$ which acts on the fixed set $[n]$, the infinite  set on which    $\pi^U_{\infty}$ acts is itself a random set. 
Moreover,  we observe that the cycles of $\pi_{\infty}$ appear in decreasing order of their smallest element, i.e. in the reverse order with respect to the usual description of the Feller coupling.
If we look at the sequence of permutations $(\pi^U_n)_{n \geq 1}$ we get a coupling of permutations of different orders, which
has the property noted in the analysis of \cite[\S 3]{MR1177897}, that 
``the cycles are built and completed one by one, in contrast to the Chinese Restaurant Process'',
with reference to the alternative construction of cycle-consistent random permutations of $[n]$ discussed in
\cite[\S 2]{MR1177897}, 
and \cite{MR2245368}. 

If the sequence of variables $(U_i)_{i \geq 1}$ is distributed like $\mathbb{P}_{\theta}$, then by  Theorem \ref{thm:np}, 
 the cycle structure of $\pi^U_{\infty}$ is directly given by a Poisson point process on $\cup_{\ell =1}^{\infty} [0,1]^{\ell}$
 with mean measure $\theta \mu(\bullet)$. In particular, the number of cycles of different lengths $\ell$ is given by independent Poisson random variables of parameter $\theta/\ell$, which generalizes the case $\theta = 1$ studied by 
 Ignatov.  

If we consider, as at the end of the introduction,  the dynamical version of our construction, where all the values of $\theta > 0$ are coupled together, 
then the evolution of the cycle structure of $\pi^U_{\infty}$ when $\theta$ varies is easy to describe in terms of Poisson processes, contrary to the case where we consider permutations of finite order. 
In particular, the set of cycles of the permutation corresponding to $\theta = \theta_1 + \theta_2$ has the same law as the union of two independent sets of cycles, corresponding to the parameters $\theta = \theta_1$ and $\theta = \theta_2$.

\section{Connection with work by Shepp and Lloyd}
\label{lloyd}

In  this  section,  we  connect  Theorem \ref{thm:np}  to  a model for a random permutation $\pi$ of a set of random size $N$, first introduced by the work of Shepp  and  Lloyd \cite{MR0195117}   on
the distribution of the lengths of the longest and shortest cycles of a uniform random permutation.  In the Shepp and Lloyd model, $N$ is assigned the geometric $(p)$ distribution $\P( N \ge n) = (1-p)^n$
for $n \ge 0$.  In a following paper \cite{MR0391216}, Balakrishnan, Sankaranarayanan, and Suyambulingom extended the model of Shepp and Lloyd to a much more general model of random permutations $\pi$ of a set of random size $N$.
For a particular choice of parameters, which was not singled out for special discussion in  \cite{MR0391216}, the model of  \cite{MR0391216} assigns $N$ a negative binomial distribution, as in the following Corollary, and given $N= n$ the
permutation $\pi$ is governed by the Ewens$(\theta)$ distribution. See also  \cite{fristedt87}, \cite{hoppe2008faa}, \cite{Hansen} (Lemma 2.1). \cite{Watterson} (Theorem 2) and \cite{Joyce} for variants of this result with different interpretations,
and further references.
\begin{corollary} \label{41}
{\em (\cite{MR0195117}, \cite{MR0391216}, \cite{fristedt87}, \cite{hoppe2008faa}) } 
Let $\theta >0$ and $p \in (0,1)$, and  let
$N(\theta, p)$ denote a random variable with the negative binomial $(\theta, p)$ distribution:
$$
\P( N(\theta , p) = n ) = \frac{ (\theta)_n }{n!} (1-p)^n p^\theta 
$$
for $n \geq 1$, which implies that $\E N(\theta, p) = \theta (1-p)/p$. 
Let $\pi$ be a permutation of random order, such that conditionally 
 given $N(\theta, p) = n$, $\pi$  has order $[n]$ and is distributed according to the Ewens$(\theta)$ measure. 
Then the number of cycles of $\pi$ of different orders $\ell$ are independent Poisson variables of parameter $(1-p)^\ell/\ell$.
\end{corollary}
\begin{proof}
Let us consider, under $\mathbb{P}_{\theta}$, the permutation $\pi_n$ of random order,  $n+1$ being the first index such that $U_{n+1} < p$. 
If we condition on the value of this index and on the order statistics of $(U_1, \dots, U_n)$, we get, from the expression of the density $d \P_{\theta}/d \P_1$, 
a permutation $\pi_n$ following Ewens distribution of parameter $\theta$. 
On the other hand, Theorem \ref{thm:np} implies that the number of cycles of different sizes in $\pi_n$ are independent Poisson variables, the expectation of the 
number of $\ell$-cycles being: 
$$\theta \mu ([p,1]^{\ell}) = \theta (1-p)^{\ell}/ \ell.$$
Hence, the corollary is proven if we show that the law of the size $n$ of the permutation is negative binomal $(\theta,p)$. Since the cycle lengths form a Poisson process with intensity proportional to $\theta$ when $p$ is fixed, 
the law of the size of the permutation in function of $\theta$ corresponds to the marginals of a L\'evy process. It is also the case for the negative binomial distribution, so it is 
enough to check that for $\theta = 1$, $n$ is geometrically distributed with parameter $p$. This fact is immediate since $n+1$ is the first time when an i.i.d. sequence of uniform variables on $[0,1]$ hits the interval $[0,p]$. 
\end{proof}
The proof above is related to the fact that if $K_{\ell}(\pi)$ is the number of $\ell$-cycles of $\pi$, 
\begin{equation}
\label{lkrepgeo}
N(\theta, p) = \sum_{\ell = 1}^\infty \ell K_{\ell}(\pi)
\end{equation}
is the canonical L\'evy decomposition of the infinitely divisible distribution of $N(\theta,p)$ as a linear combination of independent Poisson variables.
Compare with the discussion of Feller 
\cite[(2.17)]{MR0228020} who gives the well-known probability generating function of the number $K$ of cycles of a uniform random permutation $\pi_n$ of $[n]$:
\begin{equation}
\label{pgf}
\E_1 \theta^{K} = \frac{ (\theta)_n } {n! }
\end{equation}
by use of his coupling with Bernoulli$(1/i)$ variables for $1 \le i \le n$.
This comes immediately after discussion of the compound Poisson representation of the negative binomial distribution, but without 
the connection indicated in Corollary \ref{41}. This model for constructing a negative binomial variable from independent Poisson counts of cycles of
a random permutation of random size is also not mentioned in the otherwise very comprehensive account \cite{MR2032426} of models related to the Ewens sampling formula.

\providecommand{\bysame}{\leavevmode\hbox to3em{\hrulefill}\thinspace}
\providecommand{\MR}{\relax\ifhmode\unskip\space\fi MR }
\providecommand{\MRhref}[2]{%
  \href{http://www.ams.org/mathscinet-getitem?mr=#1}{#2}
}
\providecommand{\href}[2]{#2}

\bibliographystyle{plain}
\bibliography{cycles}

\begin{thebibliography}{10}

\bibitem{MR1177897}
R.~Arratia, A.~D. Barbour, and S.~Tavar\'{e}.
\newblock Poisson process approximations for the {E}wens sampling formula.
\newblock {\em Ann. Appl. Probab.}, 2(3):519--535, 1992.

\bibitem{MR2032426}
R.~Arratia, A.~D. Barbour, and S.~Tavar\'{e}.
\newblock {\em Logarithmic combinatorial structures: a probabilistic approach}.
\newblock EMS Monographs in Mathematics. European Mathematical Society (EMS),
  Z\"{u}rich, 2003.

\bibitem{MR3458588}
R.~Arratia, A.~D. Barbour, and S.~Tavar\'{e}.
\newblock Exploiting the {F}eller coupling for the {E}wens sampling formula.
\newblock {\em Statist. Sci.}, 31(1):27--29, 2016.

\bibitem{MR0391216}
V.~Balakrishnan, G.~Sankaranarayanan, and C.~Suyambulingom.
\newblock Ordered cycle lengths in a random permutation.
\newblock {\em Pacific J. Math.}, 36:603--613, 1971.

\bibitem{diac-pit86}
P.~Diaconis and J.~Pitman.
\newblock Permutations, record values and random measures.
\newblock Unpublished lecture notes of a course at Dept. Statistics., Univ.
  California, Berkeley., 1986.

\bibitem{feller1945}
W.~Feller.
\newblock The fundamental limit theorems in probability.
\newblock {\em Bull. Amer. Math. Soc.}, 51(11):800--832, 11 1945.

\bibitem{MR0228020}
W.~Feller.
\newblock {\em An introduction to probability theory and its applications.
  {V}ol. {I}}.
\newblock Third edition. John Wiley \& Sons, Inc., New York-London-Sydney,
  1968.

\bibitem{MR0272642}
D.~Foata and M.~P. Sch\"{u}tzenberger.
\newblock {\em Th\'{e}orie g\'{e}om\'{e}trique des polyn\^{o}mes
  eul\'{e}riens}.
\newblock Lecture Notes in Mathematics, Vol. 138. Springer-Verlag, Berlin-New
  York, 1970.

\bibitem{fristedt87}
B.~Fristedt.
\newblock The structure of random partitions of large sets.
\newblock Technical Report, Dept. of Mathematics, University of Minnesota,
  1987.

\bibitem{MR2351686}
A.~Gnedin and J~Pitman.
\newblock Poisson representation of a {E}wens fragmentation process.
\newblock {\em Combin. Probab. Comput.}, 16(6):819--827, 2007.

\bibitem{Hansen}
J.~C. Hansen.
\newblock A functional central limit theorem for the {E}wens sampling formula.
\newblock {\em J. Appl. Probab.}, 27(1):28--43, 1990.

\bibitem{hoppe2008faa}
F.~M. Hoppe.
\newblock Fa{\`a} di {B}runo's formula and the distributions of random
  partitions in population genetics and physics.
\newblock {\em Theoretical Population Biology}, 73(4):543--551, 2008.

\bibitem{MR617405}
Z.~Ignatov.
\newblock Point processes generated by order statistics and their applications.
\newblock In {\em Point processes and queuing problems ({C}olloq., {K}eszthely,
  1978)}, volume~24 of {\em Colloq. Math. Soc. J\'{a}nos Bolyai}, pages
  109--116. North-Holland, Amsterdam-New York, 1981.

\bibitem{Joyce}
P.~Joyce and S.~Tavar\'e.
\newblock Cycles, permutations and the structure of the {Y}ule process with
  immigration.
\newblock {\em Stochastic Processes and their Applications}, 25:309--314, 1987.

\bibitem{ktstick}
S.~V. Kerov and N.~V. Tsilevich.
\newblock Stick breaking process generated by virtual permutations with {E}wens
  distribution.
\newblock {\em J. Math. Sciences}, 87(6):4082--4093, 1997.
\newblock Translation of Zap. Nauchn. Semin. POMI, 223, 162-180 (1995).

\bibitem{MR2245368}
J.~Pitman.
\newblock {\em Combinatorial stochastic processes}, volume 1875 of {\em Lecture
  Notes in Mathematics}.
\newblock Springer-Verlag, Berlin, 2006.
\newblock Lectures from the 32nd Summer School on Probability Theory held in
  Saint-Flour, July 7--24, 2002.

\bibitem{MR0286162}
A.~R\'{e}nyi.
\newblock Th\'{e}orie des \'{e}l\'{e}ments saillants d'une suite
  d'observations.
\newblock {\em Ann. Fac. Sci. Univ. Clermont-Ferrand No.}, 8:7--13, 1962.

\bibitem{MR900810}
S.~I. Resnick.
\newblock {\em Extreme values, regular variation, and point processes},
  volume~4 of {\em Applied Probability. A Series of the Applied Probability
  Trust}.
\newblock Springer-Verlag, New York, 1987.

\bibitem{MR2744276}
J.~Sethuraman and S.~Sethuraman.
\newblock Connections between {B}ernoulli strings and random permutations.
\newblock In {\em The legacy of {A}lladi {R}amakrishnan in the mathematical
  sciences}, pages 389--399. Springer, New York, 2010.

\bibitem{MR0195117}
L.~A. Shepp and S.~P. Lloyd.
\newblock Ordered cycle lengths in a random permutation.
\newblock {\em Trans. Amer. Math. Soc.}, 121:340--357, 1966.

\bibitem{Watterson}
G.~A. Watterson.
\newblock The sampling theory of selectively neutral alleles.
\newblock {\em Advances in Applied Probability}, 6(3):463--488, 1974.

\end{thebibliography}

\end{document}